\documentclass[12pt]{amsart}
\usepackage{amsmath,amsthm,
amsopn,
amssymb}
\newtheorem{theorem}{Theorem}[section]
\newtheorem{lemma}[theorem]{Lemma}
\newtheorem{cor}[theorem]{Corollary}
\newtheorem{prop}[theorem]{Proposition}

\theoremstyle{definition}
\newtheorem{definition}[theorem]{Definition}
\newtheorem{definitions}[theorem]{Definitions}

\theoremstyle{remark}
\newtheorem{remark}[theorem]{Remark}

\newtheorem{notn}[theorem]{Notation}
\newtheorem{eg}[theorem]{Example}
\newtheorem{egs}[theorem]{Examples}

\numberwithin{equation}{section}

\DeclareMathOperator{\ham}{ham} 
 
\DeclareMathOperator{\Pz}{Z_P} 
\DeclareMathOperator{\curl}{curl} \DeclareMathOperator{\grad}{grad}
 \DeclareMathOperator{\Pspec}{P.Spec}

\DeclareMathOperator{\spec}{Spec} 
\DeclareMathOperator{\Jac}{Jac} \DeclareMathOperator{\Jacid}{J}

\newcommand{\ov}{\overline}
\newcommand{\wA}{\widehat{A}}
\newcommand{\wM}{\widehat{M}}
\newcommand{\squi}{\rightsquigarrow}

\newcommand{\C}{{\mathbb C}}

\newcommand{\N}{{\mathbb N}}

\newcommand{\Q}{{\mathbb Q}}
\newcommand{\Proj}{\C{\mathbb P}^1}

\begin{document}
\bibliographystyle{amsalpha}
\author{David A. Jordan}

\address
{School of Mathematics and Statistics\\
University of  Sheffield\\
Hicks Building\\
Sheffield S3~7RH\\
UK}

\email{d.a.jordan@sheffield.ac.uk}

\author{Sei-Qwon Oh}

\address{Department of Mathematics, Chungnam National  University, Daejeon 305-764, Korea} \email{sqoh@cnu.ac.kr}
\thanks{This work was supported by National Research Foundation of Korea Grant funded by the Korean Government 2009-0071102.}

\title[Poisson brackets and Poisson spectra]{
Poisson brackets and Poisson spectra in polynomial algebras}



\dedicatory{Dedicated to Ken Goodearl on his 65th birthday.}


\begin{abstract}
Poisson brackets on the polynomial algebra $\C[x,y,z]$ are studied. A description of all such brackets
is given and, for a significant class of Poisson brackets, the Poisson prime ideals and Poisson primitive ideals are determined.
The results are illustrated by numerous examples.
\end{abstract}

\maketitle

\section{Introduction}
The main purpose of this paper is to analyse Poisson brackets on the polynomial algebra $A:=\C[x,y,z]$ and their prime and primitive Poisson ideals. Various Poisson brackets on this algebra have appeared in the literature, for example in \cite{BGn,goodsemiclass,dajfdsPm,dajns,pich}, and several of these belong to one particular class of brackets, the \emph{Jacobian} or \emph{exact} brackets, determined by a \emph{potential} $a\in A$. For these brackets,
\[\{x,y\}=\frac{\partial a}{\partial z},\quad \{y,z\}=\frac{\partial a}{\partial x},\quad \{z,x\}=
\frac{\partial a}{\partial y}.\]
If $\{-,-\}$ is a Poisson bracket on $A$ and $b\in A$ then $b\{-,-\}$ is also a Poisson bracket, a phenomenon that does not extend to $\geq 4$ indeterminates. Poisson brackets behave well under localization and completion and we shall show that for any Poisson bracket on $A=\C[x,y,z]$ there exist a completion $\wA$ and elements
$a, b\in \wA$ such that
\[\{x,y\}=b\frac{\partial a}{\partial z},\quad \{y,z\}=b\frac{\partial a}{\partial x},\quad \{z,x\}=
b\frac{\partial a}{\partial y}.\]
This is an algebraic analogue of \cite[Theorem
5]{grabetal}. A special case, generalising the exact brackets, is obtained by taking $a=st^{-1}\in \C(x,y,z)$ and setting $b=t^2$. The main results of the paper, Theorems~ \ref{primitive} and ~\ref{Pspecexact}, determine, respectively, the Poisson primitive ideals and the Poisson prime ideals for these brackets. The non-zero Poisson prime ideals occur in two ways, each of which is geometric in nature. If $I$ is the ideal of $A$ generated by $\{x,y\}$, $\{y,z\}$ and $\{z,x\}$ then any prime ideal containing $I$ is Poisson prime. This gives rise to a Zariski-closed subset of the prime spectrum of $A$ lying within the Poisson prime ideals. This subset includes any Poisson maximal ideals of $A$. The only other source of non-zero Poisson prime ideals is the set of principal ideals of the form $f=\lambda s-\mu t$ where $(\lambda,\mu)\in \Proj$. These are always Poisson ideals and their irreducible factors $u$ generate Poisson prime ideals.
If, for $(\lambda,\mu)\in \Proj$, the polynomials $\lambda t-\mu s$  are all irreducible then the corresponding surfaces $S_f:=\{(\alpha,\beta,\gamma)\in \C^3 : f(\alpha,\beta,\gamma)=0\}$, where
$f=\lambda t-\mu s$ for some $(\lambda,\mu)\in \Proj$, form a partition of $\C^3$. Any singularities of these surfaces correspond to maximal ideals containing the above ideal $I$.
The Poisson primitive ideals are also of two types, namely the Poisson maximal ideals and the principal ideals where $u$ is an irreducible factor, with multiplicity one, of $\lambda t-\mu s$ for some $(\lambda,\mu)\in \Proj$.

These results will be illustrated, in Section 4, by a number of examples.
Many of these examples fit into the philosophy of \cite{goodsemiclass}, in that they are semiclassical limits of families of noncommutative algebras $T_q$ for which there are analogies between the prime ideals (resp. primitive ideals) of the generic instances of $T_q$ and the Poisson prime ideals (resp. Poisson primitive ideals) of $A$. However there may be prime ideals of $T_q$ that are not completely prime, in particular annihilators of finite-dimensional simple modules, and these do not fit well into this picture. The desirable behaviour, which often occurs is that there is a bijection $\phi$ the Poisson  prime ideals of $A$ and the completely prime ideals of $T_q$, for generic $q$, with both $\phi$ and $\phi^{-1}$ preserving inclusions. Although we will comment on this in individual examples we will not present full details of the noncommutative side of the picture.

Although the results are presented over $\C$, they hold for any algebraically closed field of characteristic $0$.

\begin{definitions}By a \emph{Poisson algebra} we
mean a commutative $\C$-algebra $A$ with a bilinear product
$\{-,-\}:A\times A\rightarrow A$ such that $A$ is a Lie algebra
under $\{-,-\}$ and, for all $a\in A$, $\{a,-\}$ is a
$\C$-derivation of $A$. Such a bracket will be called a \emph{Poisson
bracket} on the algebra $A$. For $a\in A$, the derivation
$\ham(a):=\{a,-\}$ is called a \emph{Hamiltonian} derivation of $A$
(or a Hamiltonian vector field). The \emph{Poisson centre},
$\Pz(A)$, of a Poisson algebra $A$ is $\{a\in A: \{a,b\}=0\mbox{ for
all }b\in A\}$.
\end{definitions}

\begin{eg}If $\mathfrak{g}$ is a complex Lie algebra then the Lie bracket on
$\mathfrak{g}$ (identified with the degree one part of
$S(\mathfrak{g})$) extends to a Poisson bracket on the symmetric
algebra $S(\mathfrak{g})$, called the Kirillov-Kostant-Souriau bracket
\cite[III.5.5]{BGl}.
\end{eg}

\begin{definitions}\label{basicdefs}
Let $\Delta$ be a set of derivations of a ring $R$. An ideal $I$ of $R$ is a $\Delta$\emph{-ideal} if $\delta(I) \subseteq I$ for all $\delta\in \Delta$.
To say that $R$
is $\Delta$\emph{-simple} means that $0$ is the only proper $\Delta$-ideal $I$ of $R$.
To say that a $\Delta$-ideal $P$ of $R$ is $\Delta$\emph{-prime} means that if $I$ and $J$ are $\Delta$-ideals such that
$IJ\subseteq P$ then $I\subseteq P$ or $J\subseteq P$.
If $\Delta=\{\delta\}$ is a singleton, we replace $\Delta$ in these definitions by $\delta$
rather than by $\{\delta\}$.
\end{definitions}

\begin{definitions}
Let $A$ be a Poisson algebra and let $\Delta$ be the set of Hamiltonian
derivations. A $\Delta$-ideal of $A$ is then called a \emph{Poisson  ideal} of $A$.
The Poisson algebra $A$ is \emph{Poisson simple} if it is
$\Delta$-simple and a $\Delta$-prime ideal $P$ of $A$ is said to be \emph{Poisson
prime}. If $I$ is a Poisson ideal of $A$ then
$A/I$ is a Poisson algebra in the obvious way:
$\{a+I,b+I\}=\{a,b\}+I$.
\end{definitions}

\begin{remark}
 In any Noetherian $\C$-algebra, the minimal primes over a $\Delta$-ideal are also $\Delta$-ideals, see \cite[3.3.3]{dix}. Consequently an ideal is  $\Delta$-prime if and only if it is both a $\Delta$-ideal and a prime ideal.
 In particular, in any Noetherian Poisson algebra, the minimal primes over a Poisson ideal are Poisson
 and an ideal is  Poisson prime if and only if it is both a Poisson ideal and a prime ideal.\end{remark}

\begin{definitions}
 We denote  the prime spectrum of a ring $R$ with the Zariski topology  by
$\spec R$. For a Poisson algebra $A$, the \emph{Poisson prime spectrum} $\Pspec A$ is the
subspace of $\spec A$ consisting of the Poisson prime ideals with the induced topology.
Thus a closed set in $\Pspec A$ has the form $V(I):=\{P\in \Pspec A:P\supseteq I\}$ for some ideal $I$ of $A$. As is observed in \cite{goodsemiclass}, replacing $I$ by the Poisson ideal it generates, $I$ can be assumed to be a Poisson ideal.
A Poisson prime ideal $P$ of $A$ is \emph{locally closed} if it is locally closed in the induced  topology, that is, it is not the intersection of the Poisson ideals that strictly contain it.
If $P$ is a Poisson prime ideal of $A$ then the Poisson bracket on $A/P$ extends to the quotient field $Q(A/P)$, by the quotient rule for derivations, and $P$ is said to be \emph{rational} if $\Pz (Q(A/P))=\C$.

\end{definitions}

\begin{definitions} By a \emph{Poisson maximal ideal} we mean a maximal ideal
that is also Poisson whereas by a \emph{maximal Poisson ideal} we
mean a Poisson ideal that is maximal in the lattice of Poisson
ideals. Similarly the term \emph{Poisson principal prime ideal} will
mean a Poisson prime ideal that is principal as an ideal, rather
than the smallest Poisson ideal containing a nominated element, and
\emph{Poisson height $n$ prime ideal} will mean a Poisson prime ideal
that has height $n$ as a prime ideal.

Let $I$ be an ideal of a Poisson algebra $A$. Following \cite[\S
3.2]{BGn}, the  \emph{Poisson core} of $I$, $\mathcal{P}(I)$, is
the largest  Poisson ideal of $A$ contained in $I$. If $I$ is prime
then $\mathcal{P}(I)$ is prime and any Poisson prime ideal that is
the Poisson core of a maximal ideal is said to be \emph{Poisson primitive}.

\end{definitions}

\begin{definition}
The ideal $J$ of a Poisson algebra $A$ generated by all elements of
the form $\{a,b\}$ where $a,b\in A$, or, equivalently, by all such
elements were $a$ and $b$ belong to a generating set for $A$,  is a
Poisson ideal such that the induced Poisson bracket on $A/J$ is
zero. We shall call a Poisson ideal $I$ of $A$ \emph{residually
null} if the induced Poisson bracket on $A/I$ is zero or,
equivalently, $J\subseteq I$. The determination of the residually
null Poisson prime ideals of $A$ is thus equivalent to the
determination of the prime spectrum of $A/J$.
\end{definition}

\begin{definitions}
If $T$ is a $\C$-algebra with a central non-unit non-zero-divisor
$t$ such that $B:=T/tT$ is commutative then there is a Poisson
bracket $\{-,-\}$ on $B$ such that $\{\ov u,\ov
v\}=\ov{t^{-1}[u,v]}$ for all $\ov u=u+tT$ and $\ov v=v+tT\in B$.
Following \cite[Chapter III.5]{BGl}, we refer to $T$ as a \emph{quantization} of the Poisson algebra $B$.  By a \emph{ deformation} of
$B$, we mean any $\C$-algebra of the form $T/(t-\lambda)T$, where
$\lambda\in \C$ is such that the central element $t-\lambda$ is a
non-unit in $T$. \end{definitions}

\begin{remark}We shall be interested in comparing the Poisson
spectrum of a Poisson algebra with the completely prime part of the prime spectrum of a deformation $U$. We shall not do this in a formal topological way partly because the notion of completely prime spectrum has two potentially different interpretations arising from possible inclusion of incompletely prime ideals in completely prime ideals. This issue does not arise when, as is often the case, all incompletely prime ideals are maximal. The determination of
the incompletely prime ideals can be significantly more technically complex than that of the completely prime ideals.
\end{remark}

\begin{remark}
\label{basicPbs} Let $A=\C[x,y,z]$. We shall be considering certain
constructions of Poisson brackets on $A$, some of which are
well-known. There are corresponding constructions of Poisson
brackets for certain overrings of $A$, in particular
for the quotient field $Q(A)=\C(x,y,z)$ of $A$ and for the $M$-adic completion $\wA$ of $A$ for a maximal ideal $M$ of $A$.
Whenever we use the notation $\wA$ it should be interpreted as such
a completion. We denote the maximal ideal $M\wA$ of $\wA$ by $\wM$.
Although these constructions of Poisson brackets for $\wA$ use
elements of $\wA$ as input data in their construction, it is
possible for $A$ to be a Poisson subalgebra of $\wA$, giving rise to
further classes of Poisson bracket on $A$.
\end{remark}

\begin{notn}
Let $w=x,y$ or $z$. We shall denote the derivation
$\partial/\partial w$ of $A$ by $\partial_w$. By \cite[Exercise
25.3]{mat}, the derivation $\partial_w$ extends to $\wA$ for all
maximal ideals $M$ of $A$. For $a\in A$ or $\wA$, we denote the
partial derivative $\partial_w(a)$ by $a_w$.
\end{notn}
\begin{notn}\label{defineFbracket}
Let $B=A$ or $\wA$ or $Q(A)$ and let $F=(f,g,h)\in B^3$. There is a bilinear
antisymmetric product $\{-,-\}^F:B\times B\rightarrow B$ such that,
for all $b,c\in B$, \[\{b,c\}^F=\begin{vmatrix}f&g&h\\
b_x&b_y&b_z\\c_x&c_y&c_z\end{vmatrix}.\] Thus $\{b,-\}^F$ is the
derivation
\[(gb_z-hb_y)\partial_x+(hb_x-fb_z)\partial_y+(fb_y-gb_x)\partial_z\]
of $B$. Note that
\begin{equation}\label{bxbybz}
\{b,x\}^F=gb_z-hb_y,\, \{b,y\}^F=hb_x-fb_z\mbox{ and }
\{b,z\}^F=fb_y-gb_x\end{equation} and that
\[\{y,z\}^F=f,\quad \{z,x\}^F=g\;\mbox{ and }\{x,y\}^F=h.\]

Any $\C$-derivation of $A$ is determined by its values on the
generators $x,y,z$. The same is true of any derivation from $A$ to $\wA$  and any such derivation is continuous, in the
$M$-adic and $\wM$-adic topologies, and so, by \cite[Exercise
25.3]{mat}, it extends uniquely to a
$\C$-derivation of $\wA$. Hence any $\C$-derivation of $\wA$ is also
determined by its values on $x,y$ and $z$. It follows that, for
$B=A$ or $\wA$, $\{-,-\}^F$ is the unique bilinear antisymmetric
product $\{-,-\}:B\times B\rightarrow B$ such that $\{y,z\}=f,
\{z,x\}=g$, $\{x,y\}=h$ and,  for all $b\in B$, $\{b,-\}$ is a
$\C$-derivation. The same conclusion holds when $B=Q(A)$, due to the quotient rule for derivations.

For $a,b,c\in A$, let
\[\Jacid_F(a,b,c)=\{a,\{b,c\}^F\}^F+\{b,\{c,a\}^F\}^F+\{c,\{a,b\}^F\}^F.\]
Thus $a,b$ and $c$ satisfy the Jacobi identity for $\{-,-\}^F$ if
and only if $\Jacid_F(a,b,c)=0$.
\end{notn}

\begin{prop}
\label{Jxyzenough} Let $B=A$ or $\wA$ or $Q(A)$ and let $F\in B^3$. Then $B$
is a Poisson algebra under $\{-,-\}^F$ if and only if
$\Jacid_F(x,y,z)=0$.
\end{prop}
\begin{proof}
As $\{-,-\}^F$ is bilinear and antisymmetric and $\{b,-\}^F$ is a
$\C$-derivation for all $b\in B$, the only issue is whether
$\Jacid_F(a,b,c)=0$ for all $a,b,c\in B$. The ``only if" part is
trivial so suppose that $\Jacid_F(x,y,z)=0$ and consider first the
case where $B=A$.

Let $S=\{a\in A:
\Jacid_F(a,b,c)=0\mbox{ for all }b,c\in \{x,y,z\}\}.$
It can be routinely checked that $S$ is a subalgebra of $A$
containing the generators $x, y$ and $z$ and hence that
$\Jacid_F(a,b,c)=0$ for all $a\in A$ and all $b,c\in \{x,y,z\}$.
Similar arguments then show successively that $\Jacid_F(a,b,c)=0$ for
all $a,b\in A$ and all $c\in \{x,y,z\}$ and that $\Jacid_F(a,b,c)=0$
for all $a,b,c\in A$.
 Thus
$\{-,-\}^F$ is a Poisson bracket on $A$.

Now consider the case $B=\wA$ and let
\[S=\{a\in \wA: \Jacid_F(a,b,c)=0\mbox{ for all }b,c\in
\{x,y,z\}\},\]
which is a subalgebra of $\wA$ containing $A$. Fix $b,c\in
\{x,y,z\}$ and let $\theta:\wA\rightarrow \wA$ be such that
$\theta(a)=\Jacid_F(a,b,c)$ for all $a\in \wA$. As
$\theta(\wM^{n+1})\subseteq \wM^n$ for all $n$, $\theta$ is
continuous for the $\wM$-adic topology. Let $a\in \wA$. Then $a=\lim
a_i$ for some Cauchy sequence $\{a_i\}$ in $A$ and $\theta(a_i)=0$
for each $i$. Hence $\theta(a)=\lim \theta(a_i)=0$ and $a\in S$.
Similar arguments apply to show that $\Jacid_F(a,b,c)=0$ for all
$a,b\in \wA$ and all $c\in \{x,y,z\}$ and then that
$\Jacid_F(a,b,c)=0$ for all $a,b,c\in \wA$.

A similar method, with the quotient rule for derivations taking the role of the
$\wM$-adic topology, applies to the case $B=Q(A)$.
\end{proof}

\begin{definition}
Let $B=A$ or $Q(A)$ or $\wA$, let $F=(f,g,h)\in B^3$ and let
$\{-,-\}^F:B\rightarrow B$ be the bilinear antisymmetric product
determined by $F$ as in \ref{defineFbracket}.  If $\{-,-\}^F$ is a
Poisson bracket, that is if $\Jacid(x,y,z)=0$, we shall say that $F$
is a \emph{Poisson triple} on $B$. In other words, $F$ is a Poisson
triple if and only if there is a Poisson bracket on $B$ such that
$\{y,z\}=f, \{z,x\}=g$ and $\{x,y\}=h$.
\end{definition}

\begin{notn}\label{gradcurlJac}
Let $B=A$ or $Q(A)$ or $\wA$, and let $F=(f,g,h)\in B^3$. We shall make use of
the functions $\grad:B\rightarrow B^3$ and $\curl:B^3\rightarrow
B^3$. Thus
\[\grad(f)=(f_x, f_y, f_z)\in B^3\mbox{ and }\curl F=(h_y-g_z,
f_z-h_x,g_x-f_y)\in B^3.\] We shall also make use of the scalar and
vector products on $B\times B$ and the Jacobian
matrix \[\Jac(f,g,h)=\begin{pmatrix}f_x&f_y&f_z\\
g_x&g_y&g_z\\h_x&h_y&h_z\end{pmatrix}.\] \end{notn}

The following formulae of vector calculus are well-known
and are easily checked in the present context. For all $f,g,h\in B$
and all $F, G\in B^3$, \begin{gather}
\label{curlgrad} \curl(\grad f)=0;\\
\label{curlmult} \curl gF=g \curl F-F\times \grad g;\\
\label{tripledot} \grad f.(\grad h\times \grad g)=|\Jac(f,g,h)|;\\
\label{vectororth} F.(F\times G)=0.
\end{gather}
Note that
$|\Jac(f,g,-)|:B\rightarrow B$ is a derivation for each pair $f,g\in
B$, being an $A$-linear combination of the derivations $h_x$, $h_y$ and $h_z$. \label{vectalgbasics}

Much of the following Proposition is well-known, particularly in the
case $B=A$. For example, (1) is in \cite[p. 252]{dml}, (2) in
\cite{Odesskii+,pich}  and (4) in \cite{Odesskii+}.
\begin{prop}
\label{baseprop} Let $B$ be $A$ or $Q(A)$ or $\wA$. Let $f,g,h,a,b\in B$ and
let $F=(f,g,h)\in B^3$.
\begin{enumerate}
\item  $F$ is a Poisson triple if and only if $F.\curl F=0$.
\item
 $\grad(a)$ is a Poisson triple on $B$.
 \item
 If $F$ is a Poisson triple on $B$ then $bF:=(bf,bg,bh)$
 is a Poisson triple on $B$.
\item $b\grad(a)$ is a Poisson triple.
\end{enumerate}\end{prop}
\begin{proof}
The proof is straightforward using Proposition~\ref{Jxyzenough} for (1), \eqref{curlgrad} and (1) for (2), \eqref{curlmult}, (1) and
\eqref{vectororth} for (3) and (2,3) for (4).
\end{proof}

\begin{definition} Let $F=(f,g,h)$ be a Poisson triple on
$B$, where $B=A$ or $Q(A)$ or $\wA$.  We say that $F$ is \emph{exact} (on $B$) if it has the form
$\grad(a)=(a_x,a_y,a_z)$ for some $a\in B$ and  \emph{m-exact} (on
$B$) if it is a multiple of an exact triple, that is if it has the
form $b\grad a=(ba_x,ba_y,ba_z)$ for some $a,b\in B$.

Let $B=A$. If there exists a maximal ideal $M$ of $A$ such that $F=b\grad
a=(ba_x,ba_y,ba_z)$ for some $a,b\in \wA$ such that
$ba_x,ba_y,ba_z\in A$, we say that $F$ is a \emph{cm-exact} bracket
on $A$. If $F=b\grad a=(ba_x,ba_y,ba_z)$ for some
$a,b\in Q(A)$ such that $ba_x,ba_y,ba_z\in A$, we say that $F$
is \emph{qm-exact}.

We shall say that a Poisson bracket on $B$ is \emph{exact}
 if the corresponding triple is exact and adopt a similar convention for
m-exact, cm-exact and qm-exact brackets.
\end{definition}

\begin{egs}
\label{cmegs} By \ref{baseprop}(4), any pair $a,b\in A$
determine an m-exact Poisson bracket on $A$. For examples of cm-exact
brackets, let us take $M$ to be the maximal ideal
$\langle x-1,y-1,z-1\rangle$ so that, with $w:=x-1$, $\wA$ contains formal
power series for $\log x=\log(1+w)=\sum_{m=1}^\infty
\frac{(-1)^{m-1}}{m} w^m$ and, for any $\alpha\in \C$, including
$\alpha=-1$, $x^\alpha=(1+w)^\alpha=\sum_{m=0}^\infty
\frac{\prod_{j=0}^{m-1}(\alpha-j)}{m!}w^m$. Similarly $\log y, \log
z, y^\alpha, z^\alpha\in \widehat A$. We can now identify two known
Poisson triples on $A$ as cm-exact.

\begin{enumerate}
\item Given $\rho, \sigma, \tau \in \C$, it is  well-known that
$(\rho yz, \sigma zx, \tau xy)$ is a Poisson triple on $A$. If
$\rho=\sigma=\tau$ this is the exact triple $\grad (\rho xyz)$
but in general it is the cm-exact triple \[xyz\grad(\rho \log
x+\sigma \log y+\tau \log z).\]

\item Let $\alpha\in \C$.  The Poisson triple $(y,-\alpha x, 0)$
features in \cite[Remark 3.6(1)]{BGn} and is discussed in detail in
\cite{goodsemiclass}. The special case $\alpha=1$ features in
\cite[Remark 3.2]{BGn}. This is the cm-exact triple
$y^{\alpha+1}\grad(xy^{-\alpha})$. Note that $z(y,-\alpha x, 0)$ is the special case of (1)
 with $\rho=1$, $\sigma=-\alpha$ and $\tau=0$ so $(y,-\alpha x, 0)$ has an alternative expression as a cm-exact triple
 $xy\grad(\log
x-\alpha \log y)$.
\end{enumerate}
\end{egs}

\begin{definitions} Two Poisson brackets $\{-,-\}_1, \{-,-\}_2$
on $A$ are said to be \emph{compatible}  if, for all $\lambda,\mu\in
\C$, the bracket $\lambda\{-,-\}_1+\mu\{-,-\}_2:A\times A\rightarrow
A$ on $A$ is a Poisson bracket. References for this definition include \cite{grabetal}.  Alternatively, two Poisson triples $F, G\in A^3$ are \emph{compatible} if, for all $\lambda,\mu\in \C$, $\lambda F+\mu G$ is a
Poisson triple.
Clearly any two exact Poisson triples on $A$ are compatible but the
same is not true for m-exact Poisson triples. The following
proposition gives a criterion, clearly not satisfied by $\grad(x)$
and $y\grad(z)$, for two m-exact or cm-exact Poisson triples to be
compatible.
\end{definitions}

\begin{prop}
\label{compat} Let $B=A$ or $Q(A)$ or $\wA$ and let $a,b,c,d\in A$ with $d\neq
0$. Then $c\grad a$ and $d\grad b$ are compatible if and only if
$|\Jac(a,b,cd^{-1})|=0$.
\end{prop}
\begin{proof}
Let $\lambda, \mu\in \C$ and let $F=\lambda c\grad(a)+\mu d\grad(b)$. Note that, by \eqref{curlmult} and
\eqref{curlgrad},
\[\curl F
=\lambda \grad(c)\times \grad(a)+\mu \grad(d)\times \grad(b).\]

By \eqref{vectororth},
\[\lambda c\grad(a).(\grad(\lambda c)\times \grad(a))=0=\mu d\grad(b).(\grad(\mu d)\times \grad(b)).\]
Hence
\begin{align*}
F.\curl F= &\lambda c\grad(a).(\grad(\mu d)\times
\grad(b))+\mu d\grad(b).(\grad(\lambda c)\times \grad(a))
\\=&\lambda\mu(c|\Jac(a,b,d)|+d|\Jac(b,a,c)|)\\
=&\lambda\mu(c|\Jac(a,b,d)|-d|\Jac(a,b,c)|)\\
=&\lambda\mu(c\delta(d)-d\delta(c)),
\end{align*}
where $\delta$ is the derivation $|\Jac(a,b,-)|:B\rightarrow B$,
which extends to the quotient field of $B$ with
$\delta(cd^{-1})=(d\delta(c)-c\delta(d))/d^2$. By
Proposition~\ref{baseprop}(1), $F$ is a Poisson triple if and only
if $\lambda\mu\delta(cd^{-1})=0$. The result follows.
\end{proof}
\begin{cor}
\label{compatcor} Let $B=A$ or $Q(A)$ or $\wA$ and let $a,b,c\in B$. Then  $c\grad a$ and $\grad b$
are compatible if and only if $|\Jac(a,b,c)|=0$.
\end{cor}
\begin{proof}
This is immediate from Theorem~\ref{compat} on setting $d=1$.
\end{proof}
\begin{remark}
Given their different roles, the symmetry in $a,b$ and $c$ in the Jacobian criterion in
Corollary~\ref{compatcor} is noteworthy. We would be interested to see a geometric
explanation or interpretation of this symmetry.
\end{remark}

\begin{remark} Donin and Makar-Limanov \cite{dml} have classified
homogeneous
Poisson triples on $A$ of degree two. Their classification includes
the cm-exact triples
\[xyz\grad(\rho \log x+\sigma \log y+\tau \log z)\] considered in
\ref{cmegs}(1) and the exact
triples $\grad f$, where $f\in A$ is
a homogeneous cubic. We shall see in the next section that all Poisson triples on $A$ are cm-exact
and so of the form $b\grad a$ for some $a,b\in \wA$. However appropriate elements $a$ and $b$ of $\wA$
are not always readily identifiable. All the triples in the classification in \cite{dml} are sums of
compatible triples $\grad b+c\grad a$ where $b,c\in A$ and $a\in Q(A)$ or $\wA$ and $a, b, c$ can be identified.
Examples, with parameters $\alpha,\delta,\lambda\in \C$, include
$\grad(\alpha xyz+\delta x^3)-yz\grad (\lambda)a$, $\grad (zk)+\lambda y^2z\grad(x/y)$, for any homogeneous quadratic $k\in A$, and
$\grad(\gamma x^2z)-\lambda x^2z\grad((y/x)+\log(z)/2)$.
\end{remark}

\section{Poisson triples on $A$}
In this section, we show that every Poisson triple on $A$ is
cm-exact. This is an algebraic analogue of \cite[Theorem
5]{grabetal} and depends on the following technical details for the
completion $\wA=\C[[x,y,z]]$ of $A$ at the maximal ideal $xA+yA+zA$.

\begin{notn}
Let $f\in\wA$ and let $i,j,k\in \N_0$. We shall denote by $f_{ijk}$
the coefficient of $x^iy^jz^k$; thus $f=\sum_{i,j,k=0}^\infty
f_{ijk}x^iy^jz^k$. If $F=(f,g,h)$ is a Poisson triple in $\wA$ then, in this notation,
Proposition~\ref{baseprop}(1) becomes
\begin{multline}\label{powerPB}0=
\left(\sum f_{abc}x^ay^bz^c\right)\left(\sum sh_{rst}x^ry^{s-1}z^t-\sum tg_{rst}x^ry^sz^{t-1}\right)\\
+\left(\sum g_{abc}x^ay^bz^c\right)\left(\sum tf_{rst}x^ry^sz^{t-1}-\sum rh_{rst}x^{r-1}y^sz^{t}\right)\\
+\left(\sum h_{abc}x^ay^bz^c\right)\left(\sum
rg_{rst}x^{r-1}y^{s}z^t-\sum sf_{rst}x^ry^{s-1}z^{t}\right).
\end{multline}
\end{notn}

\begin{lemma}\label{restdetermined}
Let $F=(f,g,h)$ and $F^\prime=(f^\prime,g^\prime,h^\prime)$ be
Poisson triples on $\wA$ with $f_{000}\neq 0$. If
$f_{ijk}=f^\prime_{ijk}$ and $h_{ijk}=h^\prime_{ijk}$ for all
$i,j,k\in \N_0$ and $g_{ij0}=g^\prime_{ij0}$ for all $i,j\in \N_0$
then $F=F^\prime$.
\end{lemma}

\begin{proof} Order
the monomials $x^iy^jz^k$ lexicographically from the right so that,
for example, $z^2>y^2z>x^2y.$ It suffices to show that
$g_{ijk}=g^\prime_{ijk}$ whenever $k\neq 0$. Suppose that this is
not the case and choose the least monomial $x^iy^jz^k$ for which $g_{ijk}\neq
g^\prime_{ijk}$. By the hypothesis and the choice of $x^iy^jz^k$, subtracting \eqref{powerPB} for $F^\prime$
 from \eqref{powerPB} for $F$ and taking coefficients of $x^iy^jz^{k-1}$ gives
 \[kf_{000}(g_{ijk}-g^\prime_{ijk})=0.\]
As $kf_{000}\neq 0$, this contradicts the choice of $x^iy^jz^k$.
\end{proof}

\begin{lemma}\label{compmJac}
Let $(f,g,h)$ be a Poisson triple in $\wA$  such that at most one of
$f, g$ and $h$ is in $\wM$. Then $F$ is m-exact.
\end{lemma}

\begin{proof}
Let $\{-,-\}$ be the Poisson bracket determined by the triple
$(f,g,h)$. Without loss of generality, we can assume that $f,g\notin
\wM$ and hence that $f_{000}\neq 0\neq g_{000}$. We seek
$b=\sum_{i,j,k=0}^\infty b_{ijk}x^iy^jz^k$ and
$d=\sum_{i,j,k=0}^\infty d_{ijk}x^iy^jz^k$ such that the following
equations are satisfied
\[
bd_x=f;\quad
bd_y=g;\quad
bd_z=h.\]
Equating coefficients of $x^iy^jz^k$, in the equations for $f$ and $h$,  and of $x^iy^j$
in the equation for $g$, we obtain
\begin{eqnarray}
\label{f}\quad f_{ijk}&=&\sum rb_{(i-r+1)(j-s)(k-t)}d_{rst},\quad 1\leq r\leq i+1, 0\leq s\leq j, 0\leq t\leq k;\\
\label{g}\quad g_{ij0}&=&\sum sb_{(i-r)(j-s+1)0}d_{rs0},
\quad 0\leq r\leq i, 1\leq s\leq j+1;\\
\label{h}\quad h_{ijk}&=&\sum
tb_{(i-r)(j-s)(k-t+1)}d_{rst}\quad 0\leq r\leq i, 0\leq s\leq j, 1\leq t\leq k+1.
\end{eqnarray}
To distinguish particular instances of these, we shall refer to \eqref{f}$_{ijk}$, \eqref{g}$_{ij0}$ and \eqref{h}$_{ijk}$.
Regard these as countably many quadratic equations in indeterminates $b_{ijk}$ and
$d_{ijk}$ with coefficients determined by the coefficients of $f, g$
and $h$. Our first aim is to show that there is a simultaneous solution to these equations.

We assign \emph{weights} to the indeterminates $b_{ijk}$ and
$d_{ijk}$ and to the equations \eqref{f}$_{ijk}$,
\eqref{g}$_{ij0}$ and \eqref{h}$_{ijk}$ as follows: $b_{ijk}$ has
weight $i+j+k$, $d_{ijk}$ has weight $i+j+k-1$, \eqref{f}$_{ijk}$ and \eqref{h}$_{ijk}$ have weight $i+j+k$
and
\eqref{g}$_{ij0}$ has weight $i+j$.

We next define a well-ordering $\rightsquigarrow$ on
the indeterminates. Let $m_1$ and $m_2$ be indeterminates with weights $w_1$ and $w_2$  respectively. If
$w_1<w_2$ then $m_1\squi m_2$ and if $w_2<w_1$ then $m_2\squi m_1$.
Now suppose that $w_1=w_2=w$. If $m_1=b_{ijk}$ and $m_2=b_{abc}$, where $i+j+k=w=a+b+c$ then $m_1\squi m_2$ if $i>a$ or if
$i=a$ and $j>b$, otherwise $m_1\squi m_2$. Similarly, if $m_1=d_{ijk}$ and $m_2=d_{abc}$, where $i+j+k=w+1=a+b+c$ then $m_1\squi m_2$ if $i>a$ or if
$i=a$ and $j>b$, otherwise $m_1\squi m_2$. Finally, if $m_1=b_{ijk}$ and $m_2=d_{abc}$, where $i+j+k=w=a+b+c-1$
then $m_2\rightsquigarrow m_1$
if and only if $i>a$. For example, when $w=2$, we have
\begin{multline*}d_{300}\squi b_{200}\squi d_{210}\squi d_{201}\squi
b_{110}\squi b_{101}\squi d_{120}\squi d_{111}\squi d_{102}\\\squi
b_{020}\squi b_{011}\squi b_{002}\squi d_{030}\squi d_{021}\squi
d_{012}\squi d_{003}.\end{multline*}

The indeterminate $d_{000}$ does
not feature in the partial derivatives of $d$ and can take
any value.
The indeterminates of weight $0$ are $d_{100}\squi b_{000}\squi d_{010}\squi d_{001}$. Set
$d_{100}=1$. By \eqref{f}$_{000}$, \eqref{g}$_{000}$ and \eqref{h}$_{000}$,
\[b_{000}=f_{000}\neq 0,\quad d_{010}=g_{000}f_{000}^{-1}\neq 0,\;\text{ and }d_{001}=h_{000}f_{000}^{-1}.\]
All the indeterminates of weight $0$ have now been determined in
such that a way that all the equations of weight $0$ are satisfied.
We now proceed by induction. Let $w>0$ and suppose that values have
been determined for all indeterminates of weight $<w$ in such that a
way that all the equations of weight $<w$ are simultaneously satisfied.

On the right hand side of each of the equations \eqref{f}, \eqref{g}, \eqref{h} of weight $w$,
each summand is a product of two indeterminates the sum of whose weights is $w$. All but two of the terms
involve only indeterminates of weight $<w$ and values for these have been
determined. Therefore we can rewrite the
equations of weight $w$ as
\begin{eqnarray}
\label{ff}f_{ijk}&=&(i+1)b_{000}d_{(i+1)jk}+b_{ijk}d_{100}+\alpha;\\
\label{gg}g_{ij0}&=&(j+1)b_{000}d_{i(j+1)0}+b_{ij0}d_{010}+\beta;\\
\label{hh}h_{ijk}&=&(k+1)b_{000}d_{ij(k+1)}+b_{ijk}d_{001}+\gamma,
\end{eqnarray}
where the values of $\alpha,\beta,\gamma\in \C$ have been determined. We assign an arbitrary value, say $0$, to
$d_{(w+1)00}$, which is the first indeterminate of weight $w$
under $\squi$. We then assign values to the other indeterminates of
weight $w$ inductively, using $\squi$. We apply \eqref{ff}$_{ijk}$
and the fact that $d_{100}\neq 0$ to determine $b_{ijk}$, given
that $d_{(i+1)jk}\squi b_{ijk}$. For $j=w-i$, we apply \eqref{gg}$_{ij0}$ and
the fact that $b_{000}\neq 0$ to determine $d_{i(j+1)0}$, given that
$b_{ij0}\squi d_{i(j+1)0}  $. For $k\neq 0$ and  $j=w-i-k$, we apply \eqref{hh}$_{ijk}$ and the fact that $b_{000}\neq 0$ to determine
$d_{ij(k+1)}$, given $b_{ijk}$. Note that, in this process, each equation \eqref{ff}$_{ijk}$ or
\eqref{gg}$_{ij0}$ or \eqref{hh}$_{ijk}$ is associated with the determination of a unique indeterminate
and so the equations of weight $w$ inherit the well-ordering given by $\squi$.
In this way values are determined for all
the indeterminates of weight $w$ in such a way that all equations of
weight $w$ are simultaneously satisfied and these are consistent with the solutions to equations of lower weight. By induction, values are determined for all
the indeterminates in such a way that all the equations \eqref{ff}$_{ijk}$, \eqref{gg}$_{ijk}$and \eqref{hh}$_{ijk}$  are simultaneously satisfied.

Let $b=\sum_{i,j,k=0}^\infty b_{ijk}x^iy^jz^k$ and
$d=\sum_{i,j,k=0}^\infty d_{ijk}x^iy^jz^k$.
Applying Lemma~\ref{restdetermined} to
$\{-,-\}$ and the m-exact bracket $b \grad d$, we see that $\{-,-\}=b \grad d$.
\end{proof}

\begin{theorem}\label{allcm}
Every Poisson triple $F=(f,g,h)$ on $\C[x,y,z]$ is cm-exact.
\end{theorem}
\begin{proof}
If two of $f,g,h$ are $0$ then $F$ is m-exact, for example
$(f,0,0)=f\grad x$. So we may assume that at least two of $f,g,h$
are non-zero. Without loss of generality, $f\neq 0$ and $g\neq 0$.
Let $M=\langle x-\alpha,y-\beta,z-\gamma\rangle$ be a maximal ideal of
$\C[x,y,z]$ such that $fg\notin M$. Changing generators, we may
assume that $\alpha=\beta=\gamma=0$. Passing to the completion
$\C[[x,y,z]]$, where $f\notin\wM$ and $g\notin\wM$, and applying
Lemma~\ref{compmJac}, we see that $F$ is cm-exact.
\end{proof}

\section{Poisson spectra for exact and $\textrm{qm}$-exact brackets}
The aim of this section is to determine the Poisson prime and primitive ideals of $A$
for a class of qm-exact Poisson brackets including the
exact brackets. The following two general results will apply.

\begin{lemma}
Let $R$ be a commutative noetherian $\C$-algebra $R$ that is a domain and let $\delta$ be a $\C$-derivation of $R$. Let
$K$ denote the subring of constants, that is $K=\{r\in R:\delta(r)=0\}$. Then $K$ is integrally closed in $R$.\label{intclosed}
\end{lemma}
\begin{proof}
Let $r\in R$ be integral over $K$ and let $n\geq 1$ be minimal such that there exist
$k_{n-1},\ldots, k_1,k_0\in K$ such that
\[r^n+k_{n-1}r^{n-1}+\ldots+k_1r+k_0=0.\] If $n=1$ then $r\in K$ so we may suppose that
$n>1$.
Then
\[0=\delta(r^n+k_{n-1}r^{n-1}+\ldots+k_1r+k_0)=\delta(r)(nr^{n-1}+(n-1)k_{n-1}r^{n-2}+\ldots+k_1).\]
As $R$ is a domain, either $\delta(r)=0$, and $r\in K$, or
$nr^{n-1}+(n-1)k_{n-1}r^{n-1}+\ldots+k_1=0$. As $n$ is invertible in $K$, the latter contradicts the minimality of $n$ so $r\in K$.
\end{proof}

\begin{prop}
Let $A$ be an affine Poisson algebra and let $P$ be a Poisson prime ideal of $A$ of codimension $d\leq 1$.
Then $P$ is residually null.\label{codimone}
\end{prop}
\begin{proof}
The result is obvious if $d=0$ so suppose that $d=1$. Let $B$ be the Poisson algebra $A/P$. It is a consequence of Noether's Normalization Theorem that there exists $w\in B$ such that $B$ is integral over $\C[w]$. For example, see \cite[14.28]{rys}. Let $b\in B$ and let $K$ and $K^\prime$ be, respectively, the subrings of constants of $\delta:=\ham w$ and $\delta^\prime:=\ham b$. Then $\C[w]\subseteq K$ so, by Lemma ~\ref{intclosed},
$B\subseteq K$. Hence $\C[w]\subseteq K^\prime$. By Lemma ~\ref{intclosed}, $B\subseteq K^\prime$.
Thus $P$ is residually null.
\end{proof}

Let $s,t\in A\backslash\{0\}$ be coprime and let $a=st^{-1}\in
Q(A)$. Then the m-exact Poisson bracket $t^2\{-,-\}_a$ on
$\C(x,y,z)$ restricts to a qm-exact Poisson bracket on $A$. For ease
of notation, we shall fix $s$ and $t$ and write this bracket simply
as $\{-,-\}$, referring to it as the \emph{qm-exact bracket determined by} $a$.
Thus \begin{equation}\label{qmb}
\{x,y\}=ts_z-st_z,\;\{y,z\}=ts_x-st_x \mbox{ and }\{z,x\}=ts_y-st_y.
\end{equation}
Note the symmetry here in the roles of $s$ and $t$. Using $-a^{-1}$ in place of $a$ yields the same Poisson
bracket on $A$. Taking $t=1$, we obtain the exact Poisson brackets
on $A$.

In the exact case, $a$ is Poisson central so $(a-\mu)A$ is a Poisson
ideal. In the next lemma, we see that, in general, the Poisson centrality of $a=st^{-1}$ in $Q(A)$ leads to a
family of principal Poisson ideals of $A$ and the minimal primes over these
are principal Poisson prime ideals.

\begin{notn}
For $(\lambda,\mu)\in \Proj$, let $f_{\lambda,\mu}=\lambda s-\mu t$.
\end{notn}

\begin{lemma}\label{htoneqmexact}
Let $(\lambda,\mu)\in \Proj$ be such that $f_{\lambda,\mu}$ is a non-zero non-unit and let $u$ be an irreducible factor in $A$ of $f_{\lambda,\mu}$.
The ideal $f_{\lambda,\mu}A$ is
Poisson and $uA$ is a Poisson prime ideal of $A$.
\end{lemma}
\begin{proof} Routine calculations show that
\begin{eqnarray*}
\{x,\lambda s-\mu t\}&=&(\lambda s-\mu t)(s_zt_y-t_zs_y),\\
\{y,\lambda s-\mu t\}&=&(\lambda s-\mu t)(s_xt_z-t_xs_z) \mbox{ and }\\
\{z,\lambda s-\mu t\}&=&(\lambda s-\mu t)(s_yt_x-t_ys_z).
\end{eqnarray*}
Hence $f_{\lambda,\mu}A$ is Poisson and, $uA$, being a minimal prime over $f_{\lambda,\mu}A$,
is Poisson prime.
\end{proof}

The next lemma specifies the Poisson maximal ideals of $A$.

\begin{lemma}\label{max}
Let $p=(\alpha,\beta,\gamma)\in \C^3$ and let $M=\langle
x-\alpha,y-\beta,z-\gamma\rangle$ be the corresponding maximal ideal of $A$.
Then $M$ is a Poisson ideal of $A$ if and only if
\begin{enumerate}
\item  $p$ is a common zero of $s$ and $t$ or
\item  $p$ is a singular point of $f_{\lambda,\mu}$ for some $(\lambda,\mu)\in\Proj$.
\end{enumerate}
\end{lemma}
\begin{proof}
Suppose that $M$ is a Poisson ideal of $A$ and that $p$ is not
a common zero of $s$ and $t$. Consider the case where $s(p)\neq0$ and $t(p)=0$.
Since
\[0=\{y,z-\gamma\}(p)=\{y,z\}(p)=(ts_x-st_x)(p)=-s(p)t_x(p),\]
we have $t_x(p)=0$. Similarly $t_y(p)=0=t_z(p)$,
so $p$ is a singular point of $f_{0,1}=-t$, and, if $s(p)=0$ and $t(p)\neq0$, then $p$ is a singular point of $f_{1,0}=s$.

Now suppose that $s(p)\neq0$ and $t(p)\neq0$, let $\lambda=t(p)/s(p)\neq0$ and let $f=f_{\lambda,1}=\lambda s-t$. Thus $f(p)=0$, $\lambda s(p)=t(p)$ and
$f_x=\lambda s_x-t_x$.
Since
\[0=\lambda\{y,z-\gamma\}(p)=\lambda\{y,z\}(p)=\lambda(ts_x-st_x)(p)
=t(p)f_x(p),\]
we have
$f_x(p)=0$. Similarly, $f_y(p)=0=f_z(p)=0$
so $p$ is a singular point of $f=f_{\lambda,1}$.

For the converse, it is clear that from \eqref{qmb} that $M$ is Poisson if  $s(p)=t(p)=0$. Suppose that $p$ is a singular point of $f_{\lambda,\mu}=\lambda s-\mu t$ for some $(\lambda,\mu)\in\Proj$. Then
\[\lambda s(p)-\mu t(p)=0=\lambda s_x(p)-\mu t_x(p)=\lambda s_y(p)-\mu t_y(p)=\lambda s_z(p)-\mu t_z(p).\]
If $\lambda\neq0$ then
\[\lambda(s_zt-st_z)(p)
=t(p)(\lambda s_z-\mu t_z)(p)=0,\]
and thus $\{x-\alpha, y\}=\{x,y\}=s_zt-st_z\in M$. Similarly
$\{x-\alpha, z\}$, $\{y-\beta, x\}$, $\{y-\beta, z\}$, $\{z-\gamma, x\}$ and $\{z-\gamma,y\}$ all
belong to  $M$, which is therefore a Poisson ideal.
The case when $\lambda=0$ and $\mu\neq0$  is similar.
\end{proof}

We next determine the Poisson primitive ideals of $A$.
For a Poisson prime ideal $P$ of an affine Poisson algebra $A$, there is, by \cite[1.7(i) and 1.10]{ohsymp}, a sequence of implications:
$$P\text{ is locally closed $\Rightarrow P$ is Poisson primitive $\Rightarrow P$ is rational}.$$
To establish the \emph{Poisson Dixmier-Moeglin equivalence}, it is enough to show that if $P$ is a rational Poisson prime ideal of $A$ then $P$ is locally closed. For further discussion of this, see \cite{goodPDM,goodsemiclass,goodlaun}. Some of the examples considered here are covered by those papers.

\begin{lemma} Let $s,t\in A\backslash\{0\}$ be coprime and let
$a=st^{-1}\in Q(A)$.
Let $(\lambda,\mu)\in\Proj$ and let $u$ be an irreducible factor in $A$ of $f_{\lambda,\mu}$.
The following are equivalent.
\begin{enumerate}
\item $uA$ is Poisson primitive;
\item $uA$ is not residually null;
\item $u$ has multiplicity one, as an irreducible factor of $f_{\lambda,\mu}$;
\item $uA$ is locally closed.
\end{enumerate}\label{primitivelemma}
\end{lemma}
\begin{proof}
First note that if $g\in A$ is such that $f_{\lambda,\mu}=ug$ then
\begin{equation*}
\lambda s=\mu t+ug\text{ and } \lambda s_x=\mu t_x+u_xg+ug_x,
\end{equation*}
from which it follows that
\begin{equation}
\lambda(st_x-ts_x)=ugt_x-u_xgt-ug_xt\text{ and }\mu(st_x-ts_x)=ugs_x-u_xsg-usg_x.
\label{either}\end{equation}

\noindent (1)$\Rightarrow$(2) Any residually null Poisson primitive ideal must be maximal.

\noindent (2)$\Rightarrow$(3) Suppose that the multiplicity of $u$ is greater than $1$. Then $u$ divides $g$ so, as $\lambda\neq 0$ or $\mu\neq 0$,
$\{y,z\}=st_x-ts_x\in uA$. Similarly $\{x,y\}, \{z,x\}\in uA$ so $uA$ is residually null.

\noindent (3)$\Rightarrow$(4) Let $Q$ be a Poisson prime ideal properly containing $uA$. By Proposition~\ref{codimone}, $Q$ is residually null so, by \eqref{qmb} and \eqref{either}, $u_xgt\in Q$, $u_xgs\in Q$ and, similarly, $u_ygt\in Q$, $u_ygs\in Q$, $u_zgt\in Q$ and
 $u_zgs\in Q$. If $uA$ were not locally closed then, as $g\notin uA$ and as $s$ and $t$, being coprime,  cannot both be in $uA$, we would have $u_y, u_x, u_z\in uA$, which is impossible, by degree, as at least one of $u_x, u_y$ and $u_z$ is non-zero.  Hence $uA$ is locally closed

\noindent (4)$\Rightarrow$(1) This holds by \cite[1.7(i)]{ohsymp}.
\end{proof}

\begin{theorem}\label{primitive} Let $s,t\in A\backslash\{0\}$ be coprime and let
$a=st^{-1}\in Q(A)$.
The Poisson primitive ideals of $A$ for the qm-exact bracket determined by $a$ are the Poisson maximal ideals, as specified in Lemma 3.5, and the ideals $uA$, where $u$ is an irreducible factor, with multiplicity one, of $f_{\lambda,\mu}$, for some $(\lambda, \mu)\in\Proj$ for which $f_{\lambda,\mu}$ is a non-zero non-unit.
Moreover $A$ satisfies the Poisson Dixmier-Moeglin equivalence.
\end{theorem}

\begin{proof} As Poisson maximal ideals are always Poisson primitive, it is immediate from Lemma~\ref{primitivelemma} that the listed ideals are all Poisson primitive. Now let $P$ be any Poisson  primitive ideal.
Since $s/t\in \Pz (Q(A))$, it follows from \cite[1.10]{ohsymp}, where Poisson primitive ideals are called \emph{symplectic}, that $P$ contains
$f_{\lambda,\mu}$ for some $(\lambda, \mu)\in\Proj$ and hence that $u\in P$
 for some irreducible factor
$u$ of $f_{\lambda,\mu}$. Either $P=uA$, in which case, by Lemma~\ref{primitivelemma}, $u$ has multiplicity one, or $P$ strictly contains $uA$, in which case,
by Proposition~\ref{codimone}, $P$ is residually null and hence maximal.

To establish the Poisson Dixmier-Moeglin equivalence, let $P^\prime$ be a rational Poisson prime ideal. If $P^\prime$ is residually null then $\C=\Pz(Q(A/P^\prime))=Q(A/P^\prime)$ and $P^\prime$ must be a Poisson maximal ideal and hence locally closed. Thus we may assume that $P^\prime$ is not residually null. By Proposition~\ref{codimone},
$P^\prime$ has height one. As $s/t\in \Pz(Q(A))$, $P^\prime$ contains an irreducible factor $u$ of $f_{\lambda,\mu}$
for some $(\lambda,\mu)\in\Proj$. As $uA$ is prime Poisson and $P^\prime$ has height one,
$P^\prime=uA$. By \ref{primitivelemma}(2)$\Rightarrow$(4), $P^\prime$ is locally closed.
Thus  $A$ satisfies
the Poisson Dixmier-Moeglin equivalence.
\end{proof}

\begin{theorem}
\label{Pspecexact} Let $s,t\in A\backslash\{0\}$ be coprime and let
$a=st^{-1}\in Q(A)$. The Poisson prime ideals for $A$
under the qm-exact bracket determined by $a$ are $0$, the residually null Poisson prime ideals
and the height one prime ideals $uA$, where $u$ is an irreducible factor of $f_{\lambda,\mu}$ for
some $(\lambda,\mu)\in \Proj$ such that $f_{\lambda,\mu}$ is a non-zero non-unit.
\end{theorem}

\begin{proof} If $a\in \C$ the Poisson bracket is $0$ and the result is trivial. So we can assume that $a\notin \C$.
It follows from Lemma~\ref{htoneqmexact} that the listed prime
ideals are indeed Poisson prime.

Now let $Q$ be a Poisson prime ideal of $A$ that is not residually
null. By Proposition~\ref{codimone} either $Q=0$ or $Q$ has height one.
 By \cite[Theorem 5.5]{mat}, $Q$ is
the intersection of the maximal ideals that contain it. As intersections of Poisson
maximal ideals are residually null, it follows that there is a
non-Poisson maximal ideal $M=\langle x-\alpha, y-\beta, z-\gamma\rangle$ such
that
 $Q\subseteq \mathcal{P}(M)\subset M$.
By Theorem~\ref{primitive}, $\mathcal{P}(M)=uA$ where $u$ is an irreducible factor, of multiplicity one, of $f_{\lambda,\mu}$ for
some $(\lambda,\mu)\in \Proj$. It follows that either $Q=0$ or $Q=uA$.
\end{proof}

In the case $t=1$ we get the following result for the exact case,
where $a$ is Poisson central.
\begin{cor}\label{Pspecexactcor}
Let $a\in A\backslash\C$. The Poisson prime ideals for $A$ under
$\{-,-\}$ are $0$, the residually null Poisson prime ideals and
the height one prime ideals $pA$, where $p$ is an irreducible factor
in $A$ of $a-\mu$ for some $\mu \in \C$.
\end{cor}

\begin{cor}\label{Pspecexactcor2}
Suppose that there are only finitely many Poisson maximal ideals of
$A$ for $\{-,-\}$.
\begin{enumerate}
\item Every residually null
Poisson prime ideal of $A$ is a Poisson maximal ideal.
\item Suppose further that $f_{\lambda,\mu}$ is irreducible for all $(\lambda,\mu)\in \Proj$.
 Then the non-zero Poisson prime
ideals of $A$ are the finitely many Poisson maximal ideals and the
ideals $f_{\lambda,\mu}A$, $(\lambda,\mu)\in \Proj$. Moreover, for $(\lambda,\mu)\in \Proj$
  and $f=f_{\lambda,\mu}$, the Poisson algebra $A/fA$ is simple  if and
only if the surface $S_{f}:=\{(x,y,z)\in \C^3:f(x,y,z)=0\}$ is
smooth.
\end{enumerate}
\end{cor}
\begin{proof}
\noindent (1) If $P$ is a residually null Poisson prime ideal that
is not maximal then, by \cite[Theorem 5.5]{mat}, $P$  is the
intersection of infinitely many maximal ideals all of which must be
Poisson.

\noindent (2) First note that $sA+tA=A$ in this case. For otherwise
$sA+tA$ is contained in a prime ideal of height at most two, by
\cite[Theorem 13.5]{mat}, and hence in infinitely maximal ideals,
all of which must be Poisson as they contain $st_w-ts_w$ for $w=x,
y$ and $z$.

Let $w=x, y$ or $z$ and let $M$ be a maximal ideal containing $f$.
Then
\[tf_w=ft_w+\lambda(s_wt-st_w)\text{ and }sf_w=fs_w+\mu(ts_w-st_w).\]
As $f\in M$, $sA+tA=A$ and $\lambda$ and $\mu$ are not both $0$, it
follows that $f_w\in M$ for $w=x,y$ and $z$ if and only if
$ts_w-st_w\in M$ for $w=x,y$ and $z$. Hence the point corresponding
to $M$ is singular if and only if $M$ is Poisson. It follows that
$S_{f}$ is smooth if and only if $fA$ is a maximal Poisson ideal.
\end{proof}

\section{Examples}

In this section we apply Theorems~\ref{primitive} and~\ref{Pspecexact} and their
corollaries to a variety of examples. For some of these the
Poisson prime and primitive ideals are already well-known but it may be instructive to view them in this more general context. In some cases we comment on
the prime spectrum of a noncommutative deformation of $A$. In the first examples, the bracket is exact, that
is, in the notation of Section~3, $t=1$ and $s=a$.

\begin{eg}
\label{firstK} Let $a=\frac{1}{2}z^2-2xy$ so that the Poisson
bracket is exact and $\{x,y\}=z$, $\{y,z\}=-2y$ and $\{z,x\}=-2x$.
As a Lie algebra under $\{-,-\}$, $\C x+\C y+\C z\simeq sl_2$. Thus
$A$ can be viewed as the symmetric algebra of $sl_2$ and $\{-,-\}$
as the Kirillov-Kostant-Souriau bracket. There is a
unique Poisson maximal ideal $M=\langle x,y,z\rangle$, $a\in M$ and, for all
$\mu\in\C$, $a-\mu$ is irreducible. By
Corollary~\ref{Pspecexactcor2}, the Poisson spectrum consists of $0$,
$M$ and the ideals $(a-\mu)A$, $\mu\in\C$, and, by
Theorem~\ref{primitive}, the Poisson primitive ideals are
all of these except $0$. In particular, for $\alpha, \beta,
\gamma\in \C$, not all zero,
$\mathcal{P}(\langle x-\alpha,y-\beta,z-\gamma\rangle)=(a-\mu)A$ where
$2\mu=\gamma^2-4\alpha\beta$. It is well-known, see
\cite[4.9.22]{dix}, that the completely prime ideals of $U=U(sl_2)$,
which is a deformation of $A$, are $0$, $(\Omega-\rho)U$, $\rho\in
\C$ and the maximal ideal $xU+yU+zU$, where $\Omega$ is the Casimir
element $4xy+z^2-2z$, and that, with the exception of $0$, they are
primitive.
\end{eg}

\begin{eg}\label{uqeg}
Let $a=2(x+y+z-xyz)$, for which $\{x,y\}=2(1-xy)$, $\{y,z\}=2(1-yz)$
and $\{z,x\}=2(1-xz)$. In the equitable presentation of
$U_q=U_q(sl_2)$, due to Ito, Terwilliger and Weng \cite{teretal},
$U_q$ is generated by $x, y$ and $z^{\pm 1}$ subject to the
relations
 \[ {q^2xy-yx}={q^2-1},\, {q^2yz-zy}={q^2-1}\mbox{
and }{q^2zx-xz}={q^2-1}.
\] Let $A^\prime=\C[x,y,z^{\pm 1}]$, the
localization of $A$ at the powers of $z$, extend $\{-,-\}$ to
$A^\prime$ and let $V_q$ be the subalgebra of $U_q$ generated by $x,
y$ and $z$. Then $U_q$ is a deformation of $A^\prime$ and $V_q$ is a
deformation of $A$. The only Poisson maximal ideals of $A$ are
$M_1:=\langle x-1,y-1,z-1\rangle$ and $M_2:=\langle x+1,y+1,z+1\rangle$, both of
which extend to $A^\prime$. By Corollary~\ref{Pspecexactcor2},
$\Pspec A=\{0,M_1,M_2\}\cup \{a-\mu:\mu\in \C\}$ and, as none of
these prime ideals contain $z$, their extensions to $A^\prime$ give
the full Poisson spectrum of $A^\prime$. For both $A$ and $A^\prime$, the Poisson primitive ideals are the non-zero Poisson prime ideals.

Using the usual presentation, as in
\cite[I.3.1]{BGl}, of $U_q(sl_2)$, it can be deduced from \cite[2.17]{htone} and
\cite[Theorem 6.1 and its proof]{primzeit} that, provided $q$ is not a root of unity, the completely prime ideals of $U_q$ are $0$, the principal ideals $(\Omega_q-\mu)U_q$, where $\Omega_q$ is the Casimir element and $\mu\in\C$, and two maximal ideals of codimension $1$, containing $\Omega_q-\mu$ for different values of $\mu$.
\end{eg}

\begin{eg}\label{locsolvinv} With $a$ as in Example~\ref{firstK}, $A/aA$ is
well-known to be the algebra of invariants for the Poisson
automorphism $\theta$ of the simple Poisson algebra $\C[x_1,y_1]$,
with $\{x_1,y_1\}=1$, such that $\theta(x_1)=-x_1$ and $\theta(y_1)=
-y_1$. In this example, and the one that follows,
 $A/aA$ again has an interpretation as the algebra of invariants for
a Poisson automorphism, of period two, of a simple
Poisson algebra.

Let $a=x(4-z^2)+y^2$ so that $\{x,y\}=-2xz$, $\{y,z\}=4-z^2$ and
$\{z,x\}=2y.$ Let $D=\C[x_1,x_2]$ with the Poisson bracket such that
$\{x_1,x_2\}=x_1$, that is, the Kirillov-Kostant-Souriau bracket for the
$2$-dimensional non-abelian solvable Lie algebra over $\C$. This
Poisson bracket extends to $B:=\C[x_1^{\pm 1},x_2]$ which is a
simple Poisson algebra and has a Poisson automorphism $\pi$ such
that $\pi(x_1)=x_1^{-1}$ and $\pi(x_2)=-x_2$. The Poisson algebra of
invariants $B^\pi$ is isomorphic to $A/aA$, see \cite{dajns}. There
are two Poisson maximal ideals in $A$, $M_1:=xA+yA+(z-2)A$ and
$M_2:=xA+yA+(z+2)A$ and $a\in M_1\cap M_2$. By
Corollary~\ref{Pspecexactcor2}, $\Pspec A=\{0,M_1,M_2\}\cup
\{(a-\mu)A:\mu\in \C\}$ and, by Theorem~\ref{primitive}, the
Poisson primitive ideals are the non-zero Poisson prime ideals.

For the Poisson bracket $\{-,-\}$, $A$ has a deformation $T$
generated by $x, y$ and $z$ subject to  the relations
\begin{eqnarray*}
xy&=&yx-2zx+3y+2z,\\
yz&=&zy-z^2+4,\\
xz&=&zx-z-2y,
\end{eqnarray*}
and with a central element $p=(4-z^2)x+y^2+3zy+z^2+4$, see
\cite{dajns}. This algebra $T$ is an iterated skew polynomial ring
$\C[z][y;\delta][x;\sigma,\delta_1]$, and standard
methods, involving \cite[Proposition
2.1.16]{McCR} and localization at the  Ore set $\{(z^2-4)^n:n\geq 0\}$, show that the completely prime ideals
of $T$ are $0$, the principal ideals $(p-\mu)T$ and two maximal ideals of codimension $1$ both containing $p-3$.

Although $aA$ corresponds topologically to $(p-3)T$,  it is $T/pT$, rather than $T/(p-3)T$, that is
the ring of invariants of the
deformation $\C[x_1^{\pm 1},x_2:x_1x_2-x_2x_1=x_1]$ of $B$, for the automorphism analogous to $\pi$. Indeed $T/pT$
is simple, in accordance with \cite[Theorem 28.3(ii)]{passman},
whereas $A/aA$, although the ring of invariants for an automorphism
of finite order of a simple Poisson algebra, is not simple.
Similar behaviour, in the context of Example~\ref{firstK} has been
observed in \cite{AF}.
\end{eg}

\begin{eg}\label{qtorusinv}
Let $a=xyz-x^2-y^2-z^2+4$ so that $\{x,y\}=xy-2z$,
$\{y,z\}=yz-2x$ and $\{z,x\}=zx-2y.$ Then $A/aA$ is isomorphic, as
a Poisson algebra, to the algebra of invariants for the Poisson
automorphism of $B:=\C[x_1^{\pm 1},x_2^{\pm 1}]$, with
$\{x_1,x_2\}=x_1x_2$, such that $\pi(x_1)=x_1^{-1}$ and
$\pi(x_2)=x_2^{-1}$, \cite{dajns}. There are five Poisson maximal
ideals $M_i$ in $A$, corresponding to the points $P_1=(2,2,2)$,
$P_2=(2,-2,-2)$, $P_3=(-2,2,-2)$, $P_4=(-2,-2,2)$, which are
singularities of $a$ and $P_5=(0,0,0)$, which is a singularity of
$a-4$.

 By
Corollary~\ref{Pspecexactcor2}, $\Pspec
A=\{0,M_1,M_2,M_3,M_4,M_5\}\cup \{(a-\mu)A:\mu\in \C\}$ and, by
Theorem~\ref{primitive}, the Poisson primitive ideals are the
non-zero Poisson prime ideals.

In \cite{dajns}, for $q\in \C$, the $\C$-algebra $T_q$ generated by
$x, y$ and $z$ subject to the relations
\begin{align*}
xy-qyx&=(1-q^2)z, \\
yz-qzy&=(q^{-1}-q)x, \\
zx-qxz&=(q^{-1}-q)y.
\end{align*} is shown to be a
deformation of $A$ and to have a central element $p$ such that
$T_q/pT_q$ is the algebra of invariants for the $\C$-automorphism  of the
quantum torus $\C[x_1^{\pm 1},x_2^{\pm 1}:x_1x_2=qx_2x_1]$ with $x_i\mapsto x_i^{-1}$ for $i=1,2$. Unlike
the previous example, the algebra $T_q$, which, if $q^2\neq 1$, is isomorphic to the
 cyclically $q$-deformed algebra $so_q(3)$ \cite{mp1}, is not directly susceptible
to skew polynomial methods and, in the case where $q$ is not a root
of unity, its prime spectrum is determined in \cite{dajtbw}. There is a bijection $\Gamma$ between
$\Pspec A$ and the set of completely prime ideals of $T_q$ and both $\Gamma$ and $\Gamma^{-1}$ preserve inclusions.
\end{eg}

\begin{eg}\label{whitney}
 Let $a=z^2-xy^2$, which determines the exact
Poisson bracket with $\{x,y\}=2z,\{y,z\}=-y^2$ and $\{z,x\}=-2xy$.
Thus $A/aA$ is the coordinate ring of the Whitney umbrella. This has a line of
singularities corresponding to the Poisson maximal ideals
$\langle x-\alpha,y,z\rangle$ for all $\alpha\in \C$. The residually null
Poisson prime ideals are those prime ideals of $A$ that contain
$yA+zA$. By Theorem~\ref{Pspecexactcor}, the Poisson spectrum
consists of $0$, the height one primes $(a-\mu)A$, $\mu\in \C$, the
height two prime ideal $yA+zA$ and the maximal ideals $\langle x-\alpha,y,z\rangle$,
$\alpha\in \C$. Of these, all except $0$ and $yA+zA$ are Poisson
primitive.
\end{eg}

\begin{eg}\label{heisenberg}
Let $a=x^2/2$ so that $\{y,z\}=x$, $\{x,y\}=\{z,x\}=0$
and $\{-,-\}$ is the Kirillov-Kostant-Souriau bracket for the
$3$-dimensional Heisenberg Lie algebra $\mathfrak{g}$. The residually null Poisson
prime ideals are the prime ideals that contain $x$. The prime ideals
generated by the irreducible factors of $x^2-\mu$, as $\mu$ varies,
are the ideals $(x-\sigma)A$, $\sigma\in \C$. These are all
Poisson and only $xA$ is residually null. By
Corollary~\ref{Pspecexactcor}, the Poisson spectrum consists of
$0$, the ideals $(x-\sigma)A$, $\sigma\in \C$, the ideals $xA+fA$,
where $f\in \C[y,z]$ is irreducible, and the Poisson maximal ideals
$\langle x,y-\beta,z-\gamma\rangle$, $\beta,\gamma\in \C$. The Poisson primitive
ideals are the ideals $(x-\sigma)A$, $\sigma\in \C\backslash\{0\}$,
and the Poisson maximal ideals. There is a homeomorphism $\Gamma:\Pspec A\rightarrow \spec U(\mathfrak{g})$,
with $\Gamma(P)$ primitive if and only if $P$ is Poisson primitive.

For a general study of Poisson prime ideals for Kirillov-Kostant-Souriau brackets of finite-dimensional soluble Lie algebras, see \cite{tauvelyu}.
\end{eg}

\begin{eg}
Let $s=x$ and $t=y$ so that $a=xy^{-1}$, $\{x,y\}=0, \{y,z\}=y, \quad \{z,x\}=-x.$
By Theorem~\ref{Pspecexact}, the Poisson spectrum of $A$
consists of $0$, $xA$, $yA$, the ideals $(x+\lambda y)$, $\lambda\in
\C$, the residually null Poisson prime ideal $xA+yA$ and the maximal
ideals $(x,y,z-\alpha)$, $\alpha\in \C$. By Theorem~\ref{primitive}, all the Poisson prime ideals, except $0$ and $xA+yA$,
are Poisson primitive. This example is documented in \cite[p199,
Remark]{BGn} where there is some inaccuracy in that the Poisson
maximal ideals are overlooked and it is stated that $xA+yA$ is
Poisson primitive.
\end{eg}

\begin{eg}
Let $\rho,\sigma,\tau\in \C$, not all $0$, and let $a=\rho x+\sigma y+\tau z$.
Thus $\{x,y\}=\tau$, $\{y,z\}=\rho$ and $\{z,x\}=\sigma$. There are no
residually null Poisson prime ideals and $\Pspec A=\{(a-\mu)A:\mu\in \C\}\cup\{0\}$.
The only Poisson prime ideal that is not Poisson primitive is $0$.
\end{eg}

\begin{eg}
Here we consider the general case where $a\in \C[x^{\pm1},y^{\pm1},z^{\pm1}]$ is a monomial. By the symmetry between $a$ and $a^{-1}$, it is enough to consider the following two cases. Let $j,k,\ell$ be non-negative integers, not all $0$, and let
\begin{enumerate}
\item $a=s=x^jy^kz^\ell$;
\item $s=x^jz^\ell$, $t=y^k$, $a=x^jy^{-k}z^\ell$.
\end{enumerate}
In (1),
\begin{eqnarray*}
\{x,y\}&=&\ell x^jy^kz^{\ell-1}=(x^{j-1}y^{k-1}z^{\ell-1})\ell xy,\\
\{y,z\}&=&jx^{j-1}y^kz^{\ell}=(x^{j-1}y^{k-1}z^{\ell-1})jyz,\mbox{ and }\\
\{z,x\}&=&kx^jy^{k-1}z^{\ell}=(x^{j-1}y^{k-1}z^{\ell-1})kzx.
\end{eqnarray*}
In (2), the same formulae for $\{x,y\}$ and $\{y,z\}$ hold but
\[
\{z,x\}=-kx^jy^{k-1}z^{\ell}=-(x^{j-1}y^{k-1}z^{\ell-1})kzx.\]
The Poisson principal prime ideals are $xA$, unless $j=0$, $yA$, unless $k=0$, $zA$, unless $\ell=0$, those of the form $(x^jy^kz^\ell-\mu)A$, $\mu\in \C\backslash\{0\}$ in Case
(1) and those of the form $(\lambda x^jz^\ell-\mu y^k)A$, $\lambda,\mu\in
\C\backslash\{0\}$, in Case (2).

Any Poisson height two prime ideal must be residually null. The height two prime $xA+yA$ is  Poisson  unless $j+k\leq1$. There are corresponding statements for $xA+zA$ and $yA+zA$. Also, if $j\geq 2$, any height two prime ideal containing $xA$ is Poisson and there are corresponding statements for $yA$ and $zA$.

In addition to $0$ and the maximal ideals containing any of the residually null Poisson primes listed, this specifies
$\Pspec A$. With the exceptions of $xA$ if $j\geq 2$, $yA$ if $k\geq 2$ and $zA$ if $\ell\geq 2$, the Poisson principal prime ideals are Poisson primitive.
\label{jkl}
\end{eg}

\begin{eg}\label{abc}
The results of Example~\ref{jkl} can be applied to the Poisson brackets of the form in Example~\ref{cmegs}(1),
where
\[\{x,y\}=\tau xy,\quad \{y,z\}=\rho yz,\quad \{z,x\}=\sigma zx,\]
where $\tau, \rho, \sigma \in \C$. These are the three-generator cases of a family of Poisson algebras that are interpreted in \cite{goodletz} as semiclassical limits of multiparameter quantum affine spaces. Their Poisson spectrum can be approached using a homeomorphism, established in \cite[Theorem 3.6]{goodletz}, with the prime spectrum of the appropriate quantum affine space.
Here, in the context of the present paper, we consider the Poisson spectrum
for this bracket in the case where $\dim_\Q(\tau\Q+\rho\Q+\sigma\Q)=1$. Other
cases will be considered later.

Renaming the indeterminates and multiplying by a non-zero scalar, we
may assume that  $\tau=\ell$, $\rho=j$ and $\sigma=\pm k$,  where $j, k$ and $\ell$ are coprime non-negative integers
and $\ell>0$.  Thus \begin{equation}\label{pmk}
\{x,y\}=\ell xy,\quad \{y,z\}=jyz,\quad
\{z,x\}=\pm kzx.\end{equation} If we denote this bracket by $\mathcal{B}_1$ then
$\mathcal{B}_2:=x^{j-1}y^{k-1}z^{\ell-1}\mathcal{B}_1$ is one of the
brackets considered in Example~\ref{jkl}. The Poisson spectrum for
$\mathcal{B}_1$ can easily be computed from that of $\mathcal{B}_2$.
For example, if $j>0, k>0$ and
$\ell>0$, then any Poisson prime ideal
for $\mathcal{B}_1$ is a Poisson prime ideal for $\mathcal{B}_2$ but
if $Q$ is a Poisson prime ideal for $\mathcal{B}_2$ then $Q$
is a Poisson prime ideal for $\mathcal{B}_1$ or $Q$ contains at
least one of $x$, $y$ and $z$. In this case, it follows from the results in Example~\ref{jkl}, that the Poisson prime spectrum for $\mathcal{B}_1$ is:
\begin{itemize}
\item $0$;
\item $xA$, $yA$ and $zA$;
\item $(x^jy^kz^\ell-\mu)A$,
$\mu\in \C\backslash\{0\}$ (if $\{z,x\}=kzx$ in \eqref{pmk});
\item $(x^jz^\ell-\mu y^k)A$, $\mu\in \C\backslash\{0\}$, (if $\{z,x\}=-kzx$ in \eqref{pmk});
\item $xA+yA$, $yA+zA$ and $xA+zA$ and the maximal ideals that contain them.
\end{itemize}
Other cases can be approached in the same way but there are cases, for example, if $k=0$ and $j=\ell=1$, where $\mathcal{B}_1$ is an
$A$-multiple of $\mathcal{B}_2$ and  cases, for example when $j=0$,
$k=2$ and $\ell=1$, where neither $\mathcal{B}_1$ nor
$\mathcal{B}_2$ is an $A$-multiple of the other.

For an alternative approach to Poisson spectra for these brackets, see \cite[9.6(b) and 9.9]{goodsemiclass}.

The examples in Example~\ref{cmegs}(2), where,
$\{x,y\}=0$, $\{y,z\}=y$ and $\{z,x\}=-\alpha x$, can be approached in a similar way
when $\alpha=m/n$ is rational, with $m$ and $n$ coprime and $n>0$. The Poisson spectrum is unchanged by multiplication by $n$, giving
\[\{x,y\}=0, \quad \{y,z\}=ny,\quad \{z,x\}=-mx.\] In the above notation, this is
$z^{-1}\mathcal{B}_1=x^{1-j}y^{1-k}\mathcal{B}_2$ with $j=n$, $\ell=0$ and $k=|m|$.  Here $\mathcal{B}_2$ is in Case (2) of Example~\ref{jkl} if $m>0$
and  Case (1) if $m<0$.  The Poisson prime spectrum is:
\begin{itemize}
\item $0$;
\item $xA$ and $yA$;
\item $(x^n-\mu y^m)A$, $\mu\in \C\backslash\{0\}$, (if $m>0$);
\item $(x^ny^{-m}-\mu)A$, $\mu\in \C\backslash\{0\}$, (if $m<0$);
\item $xA+yA$ and the maximal ideals that contain it.
\end{itemize}
\end{eg}

\begin{remark}\label{abcmore}
When $\rho,\sigma,\tau \in \C$ are such that $\dim_\Q(\rho\Q+\sigma\Q+\tau\Q)>1$  the
quadratic Poisson bracket in Example~\ref{abc} is not of the type being considered in this section.
In this case, the Poisson prime spectrum is
readily computed from the results in \cite[9.6(b)]{goodsemiclass}, including the Poisson simplicity
of $B:=\C[x^{\pm 1},y^{\pm 1},z^{\pm 1}]$ in this case.
If $\rho$, $\sigma$ and $\tau$ are all non-zero the Poisson prime ideals of $A$
are $0$, $xA$, $yA$, $zA$, the height two primes $xA+yA$, $yA+zA$
and $zA+xA$ and the maximal ideals containing any one of the Poisson
height two primes. If, for example, $\tau=0$ and $\dim_\Q(\rho\Q+\sigma\Q)>1$ then the Poisson prime ideals of $A$ are
$0$, $xA$, $yA$ and all prime ideals containing $zA$ or $xA+yA$.
In particular, taking $\tau=0, \rho=1$ and $\sigma=-\alpha\in \C\backslash\Q$,
we get the multiple, by $z$, of
the Poisson bracket from Example~\ref{cmegs}(2),
where
\[\{x,y\}=0,\quad \{y,z\}=y,\quad \{z,x\}=-\alpha x.\]
The Poisson primes for this bracket are $0$, $xA$, $yA$ and all prime ideals containing $xA+yA$.
In accordance with \cite[Example 6.4]{goodsemiclass}, all except $xA+yA$ are Poisson primitive.
\end{remark}

\end{document}